\documentclass[11pt]{amsart}
\usepackage{amsmath,amsfonts,amsthm,amsopn,amssymb,latexsym,soul}
\usepackage{array}
\usepackage{cite}
\usepackage{color,enumitem,graphicx}
\usepackage[colorlinks=true,urlcolor=blue,
citecolor=red,linkcolor=blue,linktocpage,pdfpagelabels,
bookmarksnumbered,bookmarksopen]{hyperref}
\usepackage[english]{babel}
\usepackage[left=2.61cm,right=2.61cm,top=2.73cm,bottom=2.73cm]{geometry}

\usepackage{bm}

\usepackage{lineno}
\usepackage{epsfig}
\usepackage{graphics}
\usepackage{float}
\usepackage{multirow}
\usepackage{lineno}
\usepackage{fullpage}
\usepackage[normalem]{ulem} 
\usepackage{xspace}
\usepackage{wrapfig}
\usepackage{color}
\usepackage{indentfirst}
\usepackage{comment}

\theoremstyle{definition}
\newtheorem{theorem}{Theorem}
\newtheorem{defi}{Definition}

\newtheorem{rmk}{Remark}

\newtheorem{lem}{Lemma}

\numberwithin{defi}{section} 
\numberwithin{theorem}{section}  
\numberwithin{prop}{section}  
\numberwithin{lem}{section}
\numberwithin{equation}{section} 
\numberwithin{rmk}{section}

\newcommand{\Pro}{
{\mathbf{P}}}

\newcommand{\Pas}{{\Pro\text{-a.s.}}}
\newcommand{\Fi}{{(\mathcal{F}_t )_{t\ge0}}}

\allowdisplaybreaks

\title[Exponential Decay of $L^2$-Solutions to SNLS Driven by Continuous Martingales]{Exponential Decay of $L^2$-Solutions to Stochastic Nonlinear Schr\"odinger Equations Driven by Continuous Martingales}

\author[I. D\^oku]{Isamu D\^oku}
\address{Department of Mathematics, Faculty of Education, Saitama University,
\newline\indent
Shimo-Okubo 255, Sakura-ku Saitama-shi, 338-8570, Japan}
\email{idoku@mail.saitama-u.ac.jp}

\author[S. Hashimoto]{Shunya Hashimoto}
\address{Faculty of Science, Kyoto University,
\newline\indent
Oiwake-tyou Kitashirakawa, Sakyo-ku Kyoto-shi, 606-8502, Japan}
\email{hashimoto.shunya.7m@kyoto-u.ac.jp}

\author[S. Machihara]{Shuji Machihara}
\address{Department of Mathematics, Faculty of Science, Saitama University,
\newline\indent
Shimo-Okubo 255, Sakura-ku Saitama-shi, 338-8570, Japan}
\email{machihara@rimath.saitama-u.ac.jp}

\makeatletter
\@namedef{subjclassname@2020}{\textup{2020} Mathematics Subject Classification}
\makeatother

\thanks{
This work was supported by JST, CREST Grant Number JPMJCR24Q6, Japan.
}

\begin{document}
\begin{abstract}
We investigate the global well-posedness and asymptotic behavior of $L^2$-solutions to stochastic nonlinear Schr\"odinger equations with multiplicative noise driven by continuous square integrable martingales with density. Our approach relies on a rescaling transformation that converts the stochastic system into a random nonlinear Schr\"odinger equation with a potential acting as a damping term. Unlike the standard Brownian motion case, this induced potential plays a critical role in the dynamics. We establish the global existence of solutions and prove the pathwise exponential decay of the $L^2$-norm. Crucially, the strict positivity of the decay rate is intrinsically induced by the density of the martingale\rq{}s quadratic variation. This result generalizes the stabilization known for standard Brownian motion, thereby characterizing the stabilizing effect of the martingale noise.
\end{abstract}

\maketitle

\date{}
\maketitle

\section{Introduction and Main results}

We consider the Cauchy problem for the stochastic nonlinear Schr\"odinger equation with multiplicative noise driven by a continuous square integrable martingale:
\begin{align}
\label{SNLS}
\begin{cases}
\displaystyle 
idX(t,\xi)=\Delta X(t,\xi)dt+\lambda|X(t,\xi)|^{\alpha-1}X(t,\xi)dt-\frac{i}{2}\sum_{j=1}^N|\mu_j|^2|e_j(\xi)|^2X(t,\xi)d\langle M_j\rangle(t) \\
\hspace{22mm} +iX(t,\xi)dM(t,\xi), \ t\in(0,T), \ \xi\in \mathbb{R}^d, \\
X(0,\xi)=x(\xi).
\end{cases}
\end{align}
where $\lambda=\pm1, \ \alpha>1, \ \mu_j\in \mathbb{C}, \ N<\infty$ and the driving noise $M(t,\xi)$ is defined by
\begin{align}
M(t,\xi)&=\sum_{j=1}^N\mu_je_j(\xi)M_j(t), \ t\ge 0, \ \xi\in\mathbb{R}^d.
\end{align}
Here, $M_j(t)$ are real-valued independent continuous square integrable martingales with $M_j(0)=0$ for a probability space $(\Omega,\mathcal{F},P)$ with natural filtration $(\mathcal{F}_t)_{t\ge0}$, $\langle M_j\rangle(t)$ denotes their quadratic variations, and $e_j(\xi)$ are real-valued functions. 

The nonlinear Schr\"odinger equation is a universal model for describing wave propagation in nonlinear media. We refer to \cite{BCRG94, BCRG95, BG09, BPP10} for details.
While the deterministic theory is well-established (e.g., Ginibre-Velo \cite{GV79}, Tsutsumi \cite{T87}, Kato \cite{K87, K89}, Cazenave-Weissler \cite{CW90}), physically realistic models often require random perturbations. For SNLS driven by Brownian motion, de Bouard and Debussche \cite{BD99, BD03} established the fundamental well-posedness theory. More recently, Barbu, R\"ockner, and Zhang \cite{BRZ14, BRZ16} introduced a rescaling approach to transform stochastic equations into random equations, thereby demonstrating the existence of solutions in $L^2$ and $H^1$ within the same range of exponents as in the deterministic case.
Later, the authors \cite{DHM25} applied the rescaling approach to Kato\rq{}s method \cite{K87, K89} and proved the well-posedness of solutions in $H^2$.
See also \cite{BRZ18, BHM20, BHW19, BHW22, HHM24, HHM25, HRZ19, H20} for further work on the stochastic nonlinear Schr\"odinger equation.

A central issue in the study of NLS is the asymptotic behavior of solutions. In the deterministic setting, the $L^2$-norm (mass) is typically conserved. To observe exponential decay or stabilization of the energy, one usually introduces an explicit linear damping term of the form $ibX \ (b>0)$ or considers localized damping mechanisms (see, e.g., Tsutsumi \cite{T84}, Natali \cite{N15}, and references therein). In the stochastic context, however, it is known that multiplicative noise can induce a stabilizing effect on unstable systems. The interplay between the noise intensity and the decay rate is a subject of great interest.

Our work extends the rescaling approach to the general setting of continuous martingales. Unlike the standard Brownian case, where $\langle M\rangle_t=t$, the quadratic variation here is a random process. This generalization leads to a random time-dependent potential in the rescaled equation, which presents significant mathematical challenges compared to the deterministic constant potential induced by Brownian motion. 

To rigorously address these dynamics, our main results are structured under two distinct regimes of spatial assumptions on the noise. First, to establish the global well-posedness, we require spatial decay conditions on the noise coefficients to control the nonlinear terms via Strichartz estimates. Second, to investigate the precise exponential stabilization, we replace the spatial decay with spatial homogeneity and impose a strict non-degeneracy lower bound on the noise variance. 
This setting allows us to extract the pure damping effect of the martingale on the $L^2$-norm.
Crucially, while our rescaling transformation induces random fluctuations in the $L^2$-norm, we show via the strong law of large numbers that these fluctuations are almost surely asymptotically negligible. Consequently, we establish the pathwise exponential decay of the solution, demonstrating that a strictly positive deterministic decay rate is intrinsically induced by the noise's quadratic variation, capturing the stabilizing effect without any analytical loss or constraints on the noise intensity. \\

\textbf{Assumptions and Main Results} \\

We impose the following assumptions on the coefficients and the noise:
\begin{description}
\item[\textbf{(H1)}] $e_j\in C_b^{\infty}(\mathbb{R}^d), 1\le j\le N$ such that
\[ \lim_{|\xi|\to \infty}\zeta(\xi)(|e_j(\xi)|+|\nabla e_j(\xi)|+|\Delta e_j(\xi)|)=0, \]
where 
\begin{align*}
\zeta(\xi)=
\begin{cases}
1+|\xi|^2, & d\not=2, \\
(1+|\xi|^2)(\log (3+|\xi|^2))^2, & d=2.
\end{cases}
\end{align*}
\item[\textbf{(H2)}] $\langle M_j\rangle$ have the following densities. 
\[ d\langle M_j\rangle(s)=V_j(s)ds, \Pas, \]
where $V_j(s)\in L^{r_j}(0,\infty), \ 1<r_j\le \infty$.
\item[\textbf{(H3)}] There exists some constant $C>0$ such that for any $1\le j\le N, \ \|V_j\|_{L^{\infty}(0,\infty)}\le C, \Pas,$ holds.
\end{description}
First, we define the solution of \eqref{SNLS}.
\begin{defi}
\label{sol1}
Let $\alpha\in(1,1+\frac{4}{d})$ and $T>0$. 
A $L^2$-solution of \eqref{SNLS} is a pair $(X,\tau)$, where $\tau(\le T)$ is an $(\mathcal{F}_t)$-stopping time, and $X=(X(t))_{t\in[0,\tau]}$ is a $L^2(dt)\cap L^2(d\langle M\rangle(t))$-valued continuous $(\mathcal{F}_t)$-adapted process, such that $|X|^{\alpha}\in L^1([0,T],H^{-2}(\mathbb{R}^d)), \Pas$, and it satisfies $\Pas$,
\begin{align}
\label{solW}
X(t)&=x-\int_0^t(i\Delta X(s)+\lambda i|X(s)|^{\alpha-1}X(s))ds \nonumber \\
& \quad -\frac{1}{2}\sum_{j=1}^N\int_0^t|\mu_j|^2|e_j|^2 X(s)d\langle M_j\rangle(s)+\int_0^tX(s)dM(s), \quad t\in[0,\tau],
\end{align}
as equation in $H^{-2}$.
\end{defi}
The following are the main theorems. 
First, we introduce the existence of a local $L^2$-solution of \eqref{SNLS}.
\begin{theorem}
\label{local1}
Assume (H1) and (H2), and let $1<\alpha <1+\frac{4}{d}$.
Then for any $x\in L^2$ and $0<T<\infty$, there exists a sequence of $L^2$-solutions $(X_n,\tau_n), n\in\mathbb{N}$ of \eqref{SNLS} where $(\tau_n)_{n\in\mathbb{N}}$ is a sequence of increasing stopping times.
For every $n\ge 1$, it holds $\Pas$ that
\begin{align}
X_n|_{[0,\tau_n]}\in C([0,\tau_n];L^2)\cap L^{q}(0,\tau_n;L^{\alpha+1}),
\end{align}
and uniqueness holds in $C([0,\tau_n];L^2)\cap L^{q}(0,\tau_n;L^{\alpha+1})$. Here, $(q,\alpha+1)$ is a Strichartz pair, where $q=\frac{4(\alpha+1)}{d(\alpha-1)}\in (2+\frac{4}{d},\infty).$
Moreover, defining $\displaystyle \tau^*(x)=\lim_{n\to\infty}\tau_n$ and $\displaystyle X=\lim_{n\to\infty}X_n\mathbf{1}_{[0,\tau^*(x))}$, we have the blowup alternative, that is, for $\Pas \ \omega$, if $\tau^*(x)(\omega)<T$, then
\[ \lim_{t\to \tau^*(x)(\omega)}\|X(t)(\omega)\|_{L^2}=\infty. \]
\end{theorem}
Next, we introduce the global well-posedness of the $L^2$-solution of \eqref{SNLS}.
\begin{theorem}
\label{main1}
Assume (H1), (H2) with $r_j=\infty$ for all $j$ and (H3), and let $1<\alpha<1+\frac{4}{d}$. 
Then, for each $x\in L^2$ and $0<T<\infty$, there is a unique solution $X=X(t,\xi)$ to \eqref{SNLS} in the sense of Definition \ref{sol1}, which satisfies
\begin{align}
X&\in L^2(\Omega;C([0,T];L^2)), \\
X&\in L^q(0,T;L^{\alpha+1}), \ \Pas
\end{align}
where $q=\frac{4(\alpha+1)}{d(\alpha-1)}\in (2+\frac{4}{d},\infty).$
Moreover, for $\Pas \ \omega\in\Omega$, the map $x\mapsto X(\cdot,x,\omega)$ is continuous from $L^2$ to $C([0,T];L^2)\cap L^q(0,T;L^{\alpha+1})$.
\end{theorem}
Here, we impose the following assumption in place of assumption (H1):
\begin{description}
\item[\textbf{(H4)}] The coefficients $\mu_j\in \mathbb{C}$ are not purely imaginary, i.e. $\text{Re} \ \mu_j\not=0$, and assume $M$ does not depend on the spatial variable $\xi$ (i.e., $e_j(\xi)\equiv 1$).
Furthermore, let $V_j(s)$ satisfy the following.
\begin{align*} 
V_j(s)\ge \alpha_0>0, \ \text{for all} \ s\in [0,T], \ 1\le j\le N.
\end{align*}
\end{description}
\begin{rmk}
\begin{enumerate}
\item Note that under (H4), the spatial independence $e_j(\xi)\equiv 1$ implies $\nabla M=\Delta M=0$. Consequently, the rescaled operator $A(t)$ defined in \eqref{A} reduces to the standard Laplacian $-i\Delta$. In this case, the standard Strichartz estimates apply directly without the spatial decay requirement of (H1). Thus, the global well-posedness stated in Theorem \ref{main1} remains completely valid under (H4).
\item The assumptions in (H4) are mathematically essential for our analysis due to the following reasons. To see this clearly, it is helpful to preview the rescaling transformation (introduced in Section 2), which converts the original solution $X$ into a new variable $y$ via $X=e^M y$:
\begin{enumerate}
\item[(i)] \textbf{Necessity of spatial independence}: If $e_j$ depends on the spatial variable $\xi$, the rescaled evolution operator becomes highly complicated due to the appearance of the advection term (e.g., $2\nabla M\cdot \nabla y$). Consequently, the mass $\|y(t)\|^2_{L^2}$ no longer decreases monotonically with respect to time. Furthermore, in the proof of Theorem \ref{exdecay2}, the spatial dependence prevents us from factoring out $e^{2M(t,\xi)}$ from the spatial integral $\|X(t)\|^2_{L^2}=\int_{\mathbb{R}^d}e^{2M(t,\xi)}|y(t,\xi)|^2d\xi$. This completely breaks down the subsequent asymptotic analysis. Therefore, relaxing this spatial independence assumption remains a highly challenging open problem for future research.
\item[(ii)] \textbf{Restriction on $\mu_j$}: Let $\mu_j=a_j+ib_j\in \mathbb{C}$ with $a_j,b_j\in\mathbb{R}$. In the rescaled equation for $y$, the energy dissipation rate strictly relies on $-a_j^2$. If $\mu_j$ were purely imaginary ($a_j=0$), the $L^2$-mass of the rescaled solution would be completely conserved, and no decay would occur. The condition $\text{Re} \ \mu_j\not=0$ is indispensable to guarantee energy dissipation.
\item[(iii)] \textbf{Strict positivity $V_j(t)\ge \alpha_0>0$}: The mere positivity $V_j(t)>0$ is insufficient because the density $V_j(t)$ could asymptotically decay to 0 as $t\to \infty$. In such a case, the time integral $\int_0^tV_j(s)ds$ may grow sub-linearly, failing to yield an exponential decay rate. The uniform lower bound $\alpha_0$ acts as the trigger that strictly bounds the decay rate away from zero, leading to the exponential stabilization.
\end{enumerate}
\end{enumerate}
\end{rmk}

Our third and most significant result concerns the long-time behavior of the solution. We prove that the martingale noise induces a stabilizing effect.
\begin{theorem}
\label{exdecay2}
Assume (H2) with $r_j=\infty$ for all $j$, (H3) and (H4), and let $1<\alpha<1+\frac{4}{d}$. 
Then, there exists a strictly positive deterministic constant $\omega>0$ such that, for any solution of \eqref{SNLS} given in Theorem \ref{main1}, we have
\[ \limsup_{t \to \infty} \frac{1}{t} \log \|X(t)\|^2_{L^2_{\xi}(\mathbb{R}^d)} \le -\omega, \quad \Pas \]
\end{theorem}
While the deterministic nonlinear Schr\"odinger equation strictly preserves the $L^2$-mass, precluding any exponential decay, the situation is different for general martingales. In our stochastic setting, a damping term inherently arises from the noise structure. Consequently, the solution exhibits exponential decay directly due to the influence of the noise. \\

\textbf{Organization of the Paper} \\

The remainder of this paper is organized as follows. In Section 2, we introduce the rescaling transformation and derive the random NLS with potential. In Section 3, we prove local well-posedness (Theorem \ref{local1}) using a Strichartz estimate adapted to our setting. In Section 4, we derive the mass conservation equation and prove global well-posedness (Theorem \ref{main1}). In Section 5, we show exponential decay of the $L^2$ norm of the solutions (Theorem \ref{exdecay2}).

\section{Rescaling approach}

We apply the following rescaling transformation to \eqref{SNLS}.
\begin{align}
\label{rs}
X(t,\xi)=e^{M(t,\xi)}y(t,\xi).
\end{align}
By It\^o\rq{}s product formula, we have that
\[ dX=e^Mdy+e^MydM+\frac{1}{2}\sum_{j=1}^N\mu_j^2e_j^2e^Myd\langle M_j\rangle. \]

We apply \eqref{rs} to \eqref{SNLS} to have
\begin{align}
\label{RSNLS1}
\begin{cases}
\displaystyle 
idy=e^{-M}\Delta(e^My)dt-\frac{i}{2}\sum_{j=1}^N(|\mu_j|^2|e_j|^2+\mu_j^2e_j^2)yd\langle M_j\rangle(t)+\lambda|e^{(\alpha-1)M}||y|^{\alpha-1}ydt, \\
y(0,\xi)=x(\xi).
\end{cases}
\end{align}
We rewrite equation \eqref{RSNLS1} as follows.
\begin{align}
\label{RSNLS2}
\begin{cases}
\displaystyle 
dy=A(t)ydt-\frac{1}{2}\sum_{j=1}^N(|\mu_j|^2|e_j|^2+\mu_j^2e_j^2)yd\langle M_j\rangle(t)-\lambda ie^{(\alpha-1)\text{Re}M}|y|^{\alpha-1}ydt, \\
y(0,\xi)=x(\xi),
\end{cases}
\end{align}
where
\begin{align}
\label{A}
b(t,\xi)&=2\nabla M(t,\xi), \nonumber \\
c(t,\xi)&=\sum_{j=1}^d(\partial_jM)^2+\Delta M(t,\xi), \\
A(t)y(t,\xi)&=-i(\Delta+b(t,\xi)\cdot \nabla+c(t,\xi))y(t,\xi). \nonumber
\end{align}
\begin{defi}
\label{sol2}
\begin{enumerate}
\item[(1)] 
Let $\alpha\in(1,1+\frac{4}{d})$ and $T>0$. 
A $L^2$-solution of \eqref{RSNLS1} is a pair $(y,\tau)$, where $\tau(\le T)$ is an $(\mathcal{F}_t)$-stopping time, and $y=(y(t))_{t\in[0,\tau]}$ is a $L^2(dt)\cap L^2(d\langle M\rangle(t))$-valued continuous $(\mathcal{F}_t)$-adapted process, such that $|y|^{\alpha}\in L^1([0,T];H^{-2}(\mathbb{R}^d))$, and it satisfies $\Pas$,
\begin{align*}
y(t)&=x-i\int_0^te^{-M(s)}\Delta(e^{M(s)}y(s))ds-\frac{1}{2}\sum_{j=1}^N\int_0^t(|\mu_j|^2|e_j|^2+\mu_j^2e_j^2)y(s)d\langle M_j\rangle(s) \\
& \quad -\int_0^t\lambda i |e^{(\alpha-1)M(s)}||y(s)|^{\alpha-1}y(s)ds,
\end{align*}
as equation in $H^{-2}$.
\item[(2)] 
Let $\alpha\in(1,1+\frac{4}{d})$ and $T>0$. 
A $L^2$-solution of \eqref{RSNLS2} is a pair $(y,\tau)$, where $\tau(\le T)$ is an $(\mathcal{F}_t)$-stopping time, and $y=(y(t))_{t\in[0,\tau]}$ is a $L^2(dt)\cap L^2(d\langle M\rangle(t))$-valued continuous $(\mathcal{F}_t)$-adapted process, such that $|y|^{\alpha}\in L^1([0,T];H^{-2}(\mathbb{R}^d))$, and it satisfies $\Pas$,
\begin{align*}
y(t)&=x+\int_0^tA(s)y(s)ds-\frac{1}{2}\sum_{j=1}^N\int_0^t(|\mu_j|^2|e_j|^2+\mu_j^2e_j^2)y(s)d\langle M_j\rangle(s) \\
& \quad -\int_0^t\lambda i e^{(\alpha-1)\text{Re}M(s)}|y(s)|^{\alpha-1}y(s)ds,
\end{align*}
as equation in $H^{-2}$.
\end{enumerate}
\end{defi}
The next theorem follows from an argument similar to the Brownian motion case argument \cite{BRZ14}.
\begin{lem}
\label{equiv}
\begin{enumerate}
\item[(i)] Let $(y,\tau)$ be a solution of \eqref{RSNLS1} in the sense of Definition \ref{sol2}. Set $X:=e^My$. 
Then $(X,\tau)$ is a solution of \eqref{SNLS} in the sense of Definition \ref{sol1}.
\item[(ii)] 
Let $(X,\tau)$ be a solution of \eqref{SNLS} in the sense of Definition \ref{sol1}. Set $y:=e^{-M}X$. 
Then $(y,\tau)$ is a solution of \eqref{RSNLS1} in the sense of Definition \ref{sol2}.
\end{enumerate}
\end{lem}

By Lemma \ref{equiv}, we show the following instead of Theorem \ref{main1}.
\begin{theorem}
\label{main2}
Assume (H1), (H2) with $r_j=\infty$ for all $j$ and (H3), and let $1<\alpha<1+\frac{4}{d}$. 
Then, for each $x\in L^2$ and $0<T<\infty$, there is a unique solution $y=y(t,\xi)$ to \eqref{RSNLS2} in the sense of Definition \ref{sol2}, which satisfies
\begin{align}
e^Wy&\in L^2(\Omega;C([0,T];L^2)) \\
y&\in L^q(0,T;L^{\alpha+1}), \ \Pas
\end{align}
where $q=\frac{4(\alpha+1)}{d(\alpha-1)}\in (2+\frac{4}{d},\infty).$
Moreover, for $\Pas \ \omega\in\Omega$, the map $x\mapsto y(\cdot,x,\omega)$ is continuous from $L^2$ to $C([0,T];L^2)\cap L^q(0,T;L^{\alpha+1})$.
\end{theorem}
The following lemma holds.
\begin{lem}{\cite{BRZ14}}
\label{Uts}
Assume (H1). 
For $\Pas \ \omega\in\Omega$, the operator $A(t)$ defined in \eqref{A} generates evolution operators $U(t,s)=U(t,s,\omega), \ 0\le s\le t\le T$, in the spaces $L^2(\mathbb{R}^d)$. 
Moreover, for each $x\in L^2$, the process $[s,T]\ni t\mapsto U(t,s)x\in L^2(\mathbb{R}^d)$ is continuous and $(\mathcal{F}_t)_{t\ge s}$-adapted, hence progressively measurable with respect to the filtration $(\mathcal{F}_t)_{t\ge s}$.
\end{lem}
By Lemma \ref{Uts}, the solution $y$ to \eqref{RSNLS2} is taken in the sense of a mild solution. For any $t\in[0,T]$, we obtain
\begin{align}
\label{ymild1}
y(t)&=U(t,0)x-\frac{1}{2}\sum_{j=1}^N\int_0^tU(t,s)((|\mu_j|^2|e_j|^2+\mu_j^2e_j^2)y(s))d\langle M_j\rangle (s) \nonumber \\ 
& \quad -\lambda i\int_0^t U(t,s)\left( e^{(\alpha-1)\text{Re}M(s)}|y(s)|^{\alpha-1}y(s)\right) ds. 
\end{align}
We rewrite \eqref{ymild1} under the assumption (H2) as follows.
\begin{align}
\label{ymild2}
y(t)&=U(t,0)x-\int_0^tU(t,s)\left(\frac{1}{2}\sum_{j=1}^N(|\mu_j|^2|e_j|^2+\mu_j^2e_j^2)V_j(s)y(s)+\lambda ie^{(\alpha-1)\text{Re}M(s)}|y(s)|^{\alpha-1}y(s)\right) ds ,
\end{align}
for any $t\in[0,T]$.

\section{Local existence}

The following Strichartz estimate holds.
\begin{lem}{\cite{BRZ14}}
Assume (H1). Then for any $T>0, u_0\in L^2$ and $f\in L^{q_2'}(0,T;L^{p_2'})$, the solution of
\begin{align}
u(t)=U(t,0)u_0+\int_0^tU(t,s)f(s)ds, \ 0\le t\le T,
\end{align}
satisfies the estimate
\begin{align}
\|u\|_{L^{q_1}(0,T;L^{p_1})}\le C_T\left( \|u_0\|_{L^2}+\|f\|_{L^{q_2'}(0,T;L^{p_2'})} \right),
\end{align}
where $(p_1,q_1)$ and $(p_2,q_2)$ are Strichartz pairs, namely,
\begin{align}
& (p,q)\in [2,\infty]\times [2,\infty] : \frac{2}{q}=\frac{d}{2}-\frac{d}{p} , \ \text{if} \ d\not=2, \\
& (p,q)\in [2,\infty)\times (2,\infty] : \frac{2}{q}=\frac{d}{2}-\frac{d}{p} , \ \text{if} \ d=2.
\end{align}
Furthermore, the process $C_t, t\ge0$, can be taken to be $(\mathcal{F}_t)$-progressively measurable, increasing and continuous.
\end{lem}
By Lemma \ref{equiv}, we show the following instead of Theorem \ref{local1}.
\begin{theorem}
\label{LSL^2}

Assume (H1) and (H2), and let $1<\alpha <1+\frac{4}{d}$.
Then for any $x\in L^2$ and $0<T<\infty$, there exists a sequence of $L^2$-solutions $(y_n,\tau_n), n\in\mathbb{N}$ of \eqref{RSNLS2} where $(\tau_n)_{n\in\mathbb{N}}$ is a sequence of increasing stopping times.
For every $n\ge 1$, it holds $\Pas$ that
\begin{align}
y_n|_{[0,\tau_n]}\in C([0,\tau_n];L^2)\cap L^{q}(0,\tau_n;L^{\alpha+1}),
\end{align}
and uniqueness holds in $C([0,\tau_n];L^2)\cap L^{q}(0,\tau_n;L^{\alpha+1})$. Here, $(q,\alpha+1)$ is a Strichartz pair, where $q=\frac{4(\alpha+1)}{d(\alpha-1)}\in (2+\frac{4}{d},\infty).$
Moreover, defining $\displaystyle \tau^*(x)=\lim_{n\to\infty}\tau_n$ and $\displaystyle y=\lim_{n\to\infty}y_n\mathbf{1}_{[0,\tau^*(x))}$, we have the blowup alternative, that is, for $\Pas \ \omega$, if $\tau^*(x)(\omega)<T$, then
\[ \lim_{t\to \tau^*(x)(\omega)}\|y(t)(\omega)\|_{L^2}=\infty. \]
\end{theorem}
\begin{proof}
We consider the equation in the mild sense.
\begin{align}
\label{mild}
y(t)&=U(t,0)x-\int_0^tU(t,s)\left(\frac{1}{2}\sum_{j=1}^N(|\mu_j|^2|e_j|^2+\mu_j^2e_j^2)V_j(s)y(s)+\lambda ie^{(\alpha-1)\text{Re}M(s)}g(y(s))\right) ds,
\end{align}
where $g(y)=|y|^{\alpha-1}y$. We define the following map
\begin{align}
F(y)(t)=U(t,0)x-\int_0^tU(t,s)\left(\frac{1}{2}\sum_{j=1}^N(|\mu_j|^2|e_j|^2+\mu_j^2e_j^2)V_j(s)y(s)+\lambda ie^{(\alpha-1)\text{Re}M(s)}g(y(s))\right) ds.
\end{align}
\textbf{Step}1. \ Fix $\omega \in \Omega$ and consider $F$ on the set
\[ \mathcal{X}^{\tau}_{M_1}=\left\{ y\in L^{\infty}([0,\tau];L^2)\cap L^q(0,\tau;L^{\alpha+1}): \ \|y\|_{\mathcal{X}_{\tau}}\le M_1 \right\}, \]
where $\tau=\tau(\omega)\in (0,T], \ M_1=M_1(\omega)>0$ are random variables and we define the norm as
\[ \|y\|_{\mathcal{X}_{\tau}}=\|y\|_{L^{\infty}(0,\tau;L^{2})}+\|y\|_{L^q(0,\tau;L^{\alpha+1})}. \]
From the Strichartz estimate, for any $y\in \mathcal{X}^{\tau}_{M_1}$, we have
\[ \|F(y)\|_{\mathcal{X}_{\tau}}\le 2C_{\tau}\left(\|x\|_{L^2}+\sum_{j=1}^N\|(|\mu_j|^2|e_j|^2+\mu_j^2e_j^2)V_jy\|_{L^1(0,\tau;L^{2})}+\|e^{(\alpha-1)\text{Re}M}g(y)\|_{L^{q'}(0,\tau;L^{\frac{\alpha+1}{\alpha}})}\right). \]
By H$\ddot{\text{o}}$lder\rq{}s inequality, we have
\begin{align*}
\|(|\mu_j|^2|e_j|^2+\mu_j^2e_j^2)V_jy\|_{L^1(0,\tau;L^{2})} & \le \||\mu_j|^2|e_j|^2+\mu_j^2e_j^2\|_{L^{\infty}(\mathbb{R}^d)}\|V_j\|_{L^{r_j}(0,\tau)}\|y\|_{L^{r_j'}(0,\tau;L^{2})} \\
&\le \||\mu_j|^2|e_j|^2+\mu_j^2e_j^2\|_{L^{\infty}(\mathbb{R}^d)}\|V_j\|_{L^{r_j}(0,T)}\tau^{\theta_j}\|y\|_{L^{\infty}(0,\tau;L^{2})}, \\
\|e^{(\alpha-1)\text{Re}M}g(y)\|_{L^{q'}(0,\tau;L^{\frac{\alpha+1}{\alpha}})} & \le \gamma_\tau \||y|^{\alpha}\|_{L^{q'}(0,\tau;L^{\frac{\alpha+1}{\alpha}})} \\
& \le \gamma_\tau\tau^{\widetilde{\theta}}\|y\|^{\alpha}_{L^q(0,\tau;L^{\alpha+1})},
\end{align*}
where
\[ \gamma_\tau:=\exp ((\alpha-1)\|M\|_{L^{\infty}(0,\tau;L^{\infty})}), \ \theta_j=\frac{1}{r_j'}>0, \ \widetilde{\theta}=1-\frac{d(\alpha-1)}{4}>0. \]
Therefore, setting $\theta=\min\{\theta_1,\theta_2,\cdots,\theta_N,\widetilde{\theta}\}>0$, we obtain
\begin{align}
\label{L2La}
\|F(y)\|_{\mathcal{X}_{\tau}}\le 2C_{\tau}\left(\|x\|_{L^2}+\sum_{j=1}^N\||\mu_j|^2|e_j|^2+\mu_j^2e_j^2\|_{L^{\infty}}\|V_j\|_{L^{r_j}(0,T)}\tau^{\theta}M_1+\gamma_{\tau}\tau^{\theta}M_1^{\alpha}\right).
\end{align}
Here we choose $M_1=3C_{\tau}\|x\|_{L^2}$ and take $\tau$ sufficiently small so that the following holds.
\begin{align}
\label{MTau}
2C_{\tau}\left(\|x\|_{L^2}+\sum_{j=1}^N\||\mu_j|^2|e_j|^2+\mu_j^2e_j^2\|_{L^{\infty}}\|V_j\|_{L^{r_j}(0,T)}\tau^{\theta}M_1+C_{\alpha,\tau}\gamma_{\tau}\tau^{\theta}M_1^{\alpha}\right)\le M_1,
\end{align}
where $C_{\alpha,\tau}>1$ is a constant satisfying
\begin{align*}
&\||y_1|^{\alpha-1}y_1-|y_2|^{\alpha-1}y_2\|_{L^{q'}(0,\tau;L^{\frac{\alpha+1}{\alpha}})} \\
&\le C_{\alpha,\tau}\tau^{\theta}(\|y_1\|^{\alpha-1}_{L^q(0,\tau;L^{\alpha+1})}+\|y_2\|^{\alpha-1}_{L^q(0,\tau;L^{\alpha+1})})\|y_1-y_2\|_{L^q(0,\tau;L^{\alpha+1})}.
\end{align*}
We define the real-valued continuous $\Fi$-adapted process
\[ Z_t^{(1)}:=2C_t\sum_{j=1}^N\||\mu_j|^2|e_j|^2+\mu_j^2e_j^2\|_{L^{\infty}}\|V_j\|_{L^{r_j}(0,T)}t^{\theta}+2\cdot 3^{\alpha-1}C_{\alpha,t} \|x\|^{\alpha-1}_{L^2}C_t^{\alpha}\gamma_tt^{\theta}, \ t\in[0,T], \]
and choose the $\Fi$-stopping time as
\[ \tau_1=\inf \left\{ t\in[0,T]:Z_t^{(1)}>\frac{1}{3} \right\} \wedge T. \]
Then, we have $\tau_1>0, \ Z_{\tau_1}^{(1)}\le \frac{1}{3}$, so we obtain $F(\mathcal{X}^{\tau_1}_{3C_{\tau_1}\|x\|_{L^2}})\subset \mathcal{X}^{\tau_1}_{3C_{\tau_1}\|x\|_{L^2}}$.

Next, we show that $F$ is a contraction mapping on $L^{\infty}([0,\tau_1];L^2)\cap L^q(0,\tau_1;L^{\alpha+1})$.
By a calculation analogous to that of \eqref{L2La}, we obtain
\begin{align}
\label{2/3}
& \|F(y_1)-F(y_2)\|_{\mathcal{X}_{\tau_1}} \nonumber \\
& \le 2C_{\tau_1}\sum_{j=1}^N\|(|\mu_j|^2|e_j|^2+\mu_j^2e_j^2)V_j(y_1-y_2)\|_{L^1(0,\tau_1;L^2)} \nonumber \\
& \quad +2C_{\tau_1}\gamma_{\tau_1}\||y_1|^{\alpha-1}y_1-|y_2|^{\alpha-1}y_2\|_{L^{q'}(0,\tau_1;L^{\frac{\alpha+1}{\alpha}})} \nonumber \\
& \le 2C_{\tau_1}\sum_{j=1}^N\||\mu_j|^2|e_j|^2+\mu_j^2e_j^2\|_{L^{\infty}}\|V_j\|_{L^{r_j}(0,T)}\tau_1^{\theta}\|y_1-y_2\|_{L^{\infty}(0,\tau_1;L^2)} \nonumber \\
& \quad +2C_{\alpha,\tau_1} C_{\tau_1}\gamma_{\tau_1}\tau_1^{\theta}(\|y_1\|^{\alpha-1}_{L^q(0,\tau_1;L^{\alpha+1})}+\|y_2\|^{\alpha-1}_{L^q(0,\tau_1;L^{\alpha+1})})\|y_1-y_2\|_{L^q(0,\tau_1;L^{\alpha+1})} \nonumber \\
& \le \left(4C_{\tau_1}\sum_{j=1}^N\||\mu_j|^2|e_j|^2+\mu_j^2e_j^2\|_{L^{\infty}}\|V_j\|_{L^{r_j}(0,T)}\tau_1^{\theta}+4C_{\alpha,\tau_1} C_{\tau_1}\gamma_{\tau_1}\tau_1^{\theta}M_1^{\alpha-1}\right)\|y_1-y_2\|_{\mathcal{X}_{\tau_1}} \nonumber \\
& = 2Z_{\tau_1}^{(1)}\|y_1-y_2\|_{\mathcal{X}_{\tau_1}} \nonumber \\
& \le \frac{2}{3}\|y_1-y_2\|_{\mathcal{X}_{\tau_1}},
\end{align}
for $y_1,y_2\in \mathcal{X}^{\tau_1}_{3C_{\tau_1}\|x\|_{L^2}}$.
Then, $F$ is a contraction on the space $L^{\infty}([0,\tau_1];L^2)\cap L^q(0,\tau_1;L^{\alpha+1})$. 
Hence, by Banach\rq{}s fixed point theorem, we obtain that there exists a unique solution $y\in L^{\infty}([0,\tau_1];L^2)\cap L^q(0,\tau_1;L^{\alpha+1})$ satisfying $y=F(y)$ on $[0,\tau_1]$. Moreover, there exists a sequence $y_{1,m}\in L^{\infty}(0,\tau_1;L^2)\cap L^q(0,\tau_1;L^{\alpha+1}), \ m\in \mathbb{N}$ such that $y_{1,m+1}=F(y_{1,m}), \ m\ge 1$, $y_{1,1}(t)=U(t,0)x$ and $\displaystyle \lim_{m\to \infty}y_{1,m}|_{[0,\tau_1]}=y$ in $L^{\infty}([0,\tau_1];L^2)\cap L^q(0,\tau_1;L^{\alpha+1})$. Define $y_1(t):=y(t\wedge \tau_1), \ t\in [0,T]$. Then we have 
\[ y_1=\lim_{m\to \infty}u_{1,m}(\cdot \wedge \tau_1), \ \text{in} \ L^{\infty}([0,T];L^2). \]
Therefore, since each $y_{1,m}$ is $\Fi$-adapted, $y_1$ is also $\Fi$-adapted.
So, we deduce that $(y_1,\tau_1)$ is a $L^2$-solution of \eqref{RSNLS2} such that $y_1(t)=y_1(t\wedge \tau_1), \ t\in [0,T]$ and $y_1|_{[0,\tau_1]}\in L^{\infty}([0,\tau_1];L^2)\cap L^q(0,\tau_1;L^{\alpha+1})$. \\

\noindent
\textbf{Step}2. \ 
We use an induction argument. Suppose that at the $n$-th step we have an $L^2$-solution $(y_n,\tau_n)$ such that $\tau_n\ge \tau_{n-1}, \ y_n(t)=y_n(t\wedge \tau_n), \ t\in [0,T], \ y_n|_{[0,\tau_n]}\in L^{\infty}([0,\tau_n];L^2)\cap L^q(0,\tau_n;L^{\alpha+1})$. 
We construct $(y_{n+1},\tau_{n+1}), \ \tau_{n+1}\ge \tau_n$.
Set
\[ \mathcal{X}^{\sigma_n}_{M_{n+1}}=\left\{ z\in L^{\infty}([0,\sigma_n];L^2)\cap L^q(0,\sigma_n;L^{\alpha+1}): \ \|z\|_{\mathcal{X}_{\sigma_n}}\le M_{n+1} \right\}, \]
where $\sigma_n=\sigma_n(\omega), \ M_{n+1}=M_{n+1}(\omega)$ are random variables. Then, we define the map $F_n$ as
\begin{align}
F_n(z)(t)&=U(\tau_n+t,\tau_n)y_n(\tau_n)-\int_0^tU(\tau_n+t,\tau_n+s) \nonumber \\
& \quad \times\left\{\frac{1}{2}\sum_{j=1}^N(|\mu_j|^2|e_j|^2+\mu_j^2e_j^2)V_j(s)z(s)+\lambda ie^{(\alpha-1)\text{Re}M(\tau_n+s)}g(z(s)))\right\}ds.
\end{align}
Similarly to \eqref{L2La}, for any $z\in \mathcal{X}^{\sigma_n}_{M_{n+1}}$, we have
\begin{align}
\|F_n(z)\|_{\mathcal{X}_{\sigma_n}}&\le 2C_{\tau_n+\sigma_n}\left(\|y_n(\tau_n)\|_{L^2} \right. \nonumber \\
& \quad \left.+\sum_{j=1}^N\||\mu_j|^2|e_j|^2+\mu_j^2e_j^2\|_{L^{\infty}}\|V_j\|_{L^{r_j}(0,T)}\sigma_n^{\theta}M_{n+1}+\gamma_{\tau_n+\sigma_n}\sigma_n^{\theta}M^{\alpha}_{n+1}\right).
\end{align}
This implies that $F(\mathcal{X}^{\sigma_n}_{M_{n+1}})\subset \mathcal{X}^{\sigma_n}_{M_{n+1}}$ and that $F_n$ is a contraction map on $\mathcal{X}^{\sigma_n}_{M_{n+1}}$. Here, we choose $M_{n+1}=3C_{\tau_n+\sigma_n}\|y_n(\tau_n)\|_{L^2}$ and take $\sigma_n$ sufficiently small so that the following holds.
\begin{align*}
2C_{\tau_n+\sigma_n}\left(\|y_n(\tau_n)\|_{L^2}+\sum_{j=1}^N\||\mu_j|^2|e_j|^2+\mu_j^2e_j^2\|_{L^{\infty}}\|V_j\|_{L^{r_j}(0,T)}\sigma_n^{\theta}M_{n+1}\right. \\
\left. +C_{\alpha,\tau_n+\sigma_n}\gamma_{\tau_n+\sigma_n}\sigma_n^{\theta}M^{\alpha}_{n+1}\right)\le M_{n+1}.
\end{align*}
Similarly to Step 1, we define the real-valued continuous $\{ \mathcal{F}_{\tau_n+t} \}$-adapted process
\[ Z_t^{(n+1)}:=2C_{\tau_n+t}\sum_{j=1}^N\||\mu_j|^2|e_j|^2+\mu_j^2e_j^2\|_{L^{\infty}}\|V_j\|_{L^{r_j}(0,T)}t^{\theta}+2\cdot 3^{\alpha-1}C_{\alpha,\tau_n+t}\|y_n(\tau_n)\|^{\alpha-1}_{L^2}C^{\alpha}_{\tau_n+t}\gamma_{\tau_n+t}t^{\theta}, \]
and choose the $\{ \mathcal{F}_{\tau_n+t} \}$-stopping time as
\[ \sigma_n:=\inf \left\{ t\in [0,T-\tau_n]:Z_t^{(n+1)}>\frac{1}{3} \right\} \wedge (T-\tau_n). \]
Then, $\sigma_n>0$ and $Z_{\sigma_n}^{(n+1)}\le \frac{1}{3}$.

Let $\tau_{n+1}:=\tau_n+\sigma_n$. Then $\tau_{n+1}$ is a $\Fi$-stopping time.
Similarly to Step 1, by the fixed point theorem, there exists a unique $z_{n+1}\in \mathcal{X}^{\sigma_n}_{M_{n+1}}$ satisfying $z_{n+1}=F_n(z_{n+1})$.
We define
\begin{align}
y_{n+1}(t)=
\begin{cases}
y_n(t) & t\in [0,\tau_n], \\
z_{n+1}((t-\tau_n)\wedge \sigma_n) & t\in (\tau_n,T].
\end{cases}
\end{align}
It follows from the definition of $F$ and $F_n$ that $y_{n+1}=F(y_{n+1})$ on $[0,\tau_{n+1}]$.
This implies that $y_{n+1}$ satisfies \eqref{mild} on $[0,\tau_{n+1}]$.
And, similar to the proof of \cite{BRZ14}, $y_{n+1}$ is $\Fi$-adapted in $L^2$.
Hence, $(y_{n+1},\tau_{n+1})$ is a $L^2$-solution of \eqref{RSNLS2}, such that $\tau_{n+1}\ge \tau_n, \ y_{n+1}(t)=y_{n+1}(t\wedge \tau_{n+1}), \ t\in[0,T], \ y_{n+1}|_{[0,\tau_{n+1}]}\in L^{\infty}([0,\tau_{n+1}];L^2)\cap L^q(0,\tau_{n+1};L^{\alpha+1})$. \\

\noindent
\textbf{Step}3. \ 
Next, we show that for each $n\in\mathbb{N}$, $y_n$ is continuous on $[0,\tau_n]$. First, we prove continuity at $t=0$. For an initial value $x\in L^2$ and $t\le \tau_n$, using the same argument as in Step 1, we have
\begin{align*}
&\|y_n(t)-x\|_{L^2} \\
&\le \|U(t,0)x-x\|_{L^2}+\left\|\int_0^tU(t-s)\left(\frac{1}{2}\sum_{j=1}^N(|\mu_j|^2|e_j|^2+\mu_j^2e_j^2)V_j(s)y_n(s)+\lambda ie^{(\alpha-1)\text{Re}M(s)}g(y_n(s))\right)ds\right\|_{L^2} \\
&\le \|U(t,0)x-x\|_{L^2}+\widetilde{C}_t|t|^{\theta}M^p\to 0, \quad \text{as} \ t\to 0.
\end{align*}
Note that the constant $\widetilde{C}_t$ satisfies $\widetilde{C}_t\to 0$ as $t\to 0$.
For general values of $t$, one may proceed similarly by taking $t$ as the initial time and $y_n(t)$ as the initial value. \\

\noindent
\textbf{Step}4. \ 
To show uniqueness, we take two solutions $(\tilde{y}_i,\sigma_i), \ i=1,2$. By \eqref{2/3}, for any $t,s>0$ with $s+t<\sigma_1 \wedge \sigma_2$, we have
\begin{align*}
& \|\tilde{y}_1-\tilde{y}_2\|_{L^{\infty}(s,s+t;L^2)}+\|\tilde{y}_1-\tilde{y}_2\|_{L^q(s,s+t;L^{\alpha+1})} \\
& \le \left(2C_t\sum_{j=1}^N\||\mu_j|^2|e_j|^2+\mu_j^2e_j^2\|_{L^{\infty}}\|V_j\|_{L^{r_j}(0,T)}t^{\theta}+4C_{\alpha,t} C_t\gamma_tt^{\theta}M^{\alpha-1}\right) \\
& \quad \times (\|\tilde{y}_1-\tilde{y}_2\|_{L^{\infty}(s,s+t;L^2)}+\|\tilde{y}_1-\tilde{y}_2\|_{L^q(s,s+t;L^{\alpha+1})}),
\end{align*}
where $M=\|\tilde{y}_1\|_{L^q(s,s+t;L^{\alpha+1})}+\|\tilde{y}_2\|_{L^q(s,s+t;L^{\alpha+1})}<\infty \ \text{a.s.}$.
Therefore, by taking $t$ sufficiently small, we have $\tilde{y}_1=\tilde{y}_2$ on $[s,s+t]$, which implies uniqueness on $[0,\sigma_1\wedge \sigma_2)$. Furthermore, by continuity on $L^2$, we have $\tilde{y}_1=\tilde{y}_2$ on $[0,\sigma_1\wedge \sigma_2]$.

Finally, we prove the blowup alternative. Suppose that 
\[ \Pro(M^*<\infty;\tau_n<\tau^*(x),\forall n\in \mathbb{N})>0, \]
where 
\[ M^*:=\sup_{t\in [0,\tau^*(x))}\|y(t)\|_{L^2}. \]
We define the real-valued continuous $\Fi$-adapted process
\[ Z_t:=2C_{\tau^*(x)+t}\sum_{j=1}^N\||\mu_j|^2|e_j|^2+\mu_j^2e_j^2\|_{L^{\infty}}\|V_j\|_{L^{r_j}(0,T)}t^{\theta}+2\cdot 3^{\alpha-1}C_{\alpha,\tau^*(x)+t}(M^*)^{\alpha-1}C^{\alpha}_{\tau^*(x)+t}\gamma_{\tau^*(x)+t}t^{\theta}, \]
and choose the $\Fi$-stopping time as
\[ \sigma:=\inf \left\{ t\in [0,T]:Z_t>\frac{1}{6} \right\} \wedge T. \]
For $\omega \in \{ M^*<\infty;\tau_n<\tau^*(x),\forall n\in \mathbb{N} \}$, since $\sigma_n(\omega)<T-\tau_n(\omega)$, we have
\[ \sigma_n(\omega)=\inf \{ t\in[0,T-\tau_n]:Z_t^{(n)}(\omega)>\frac{1}{3} \}. \]
On the other hand, since $\|y(\tau_n)\|_{L^2}\le M^*, \ C_{\tau_n+t}\le C_{\tau^*(x)+t}, \ \gamma_{\tau_n+t}\le \gamma_{\tau^*(x)+t}$ for any $n\ge 1$, $Z_t\ge Z_t^{(n)}$ holds. Therefore, $\sigma_n(\omega)>\sigma(\omega)>0$. Hence, $\tau_{n+1}(\omega)=\tau_n(\omega)+\sigma_n(\omega)>\tau_n(\omega)+\sigma(\omega)$. This implies $\tau_{n+1}(\omega)>\tau_1(\omega)+n\sigma(\omega)$ for every $n\ge 1$. Therefore, after a finite number of iterations, $\tau_n(\omega)$ exceeds $T$. This contradicts $\tau_n(\omega)\le T$. This completes the proof.
\end{proof}

\section{Global existence}

\begin{theorem}
\label{Global1}

Assume (H1),(H2) with $r_j=\infty$ for all $j$ and (H3).
Let $x\in L^2, \alpha\in(1,1+\frac{4}{d})$ and $(X,(\tau_n)_{n\in \mathbb{N}},\tau^*(x))$ be the maximal local solution of \eqref{SNLS}. 
We have $\Pas$ for $0\le t<\tau^*(x)$,
\begin{align}
\label{Mc}
\|X(t)\|^2_{L^2}=\|x\|^2_{L^2}+2\sum_{k=1}^N\int_0^t\int_{\mathbb{R}^d}\text{Re} (\mu_k) e_k|X(s)|^2d\xi dM_k(s).
\end{align}
Moreover
\begin{align}
\label{AE}
\mathbb{E}\left[ \sup_{0\le t<\tau^*(x)}\|X(t)\|^2_{L^2}\right]\le \widetilde{C}(T)<\infty.
\end{align}
\end{theorem}
\begin{proof}
Let $\{f_j\}_{j\in\mathbb{N}}\subset H^2(\mathbb{R}^d)$ be an orthonormal basis in $L^2$.
We define $J_{\varepsilon}=(I-\varepsilon \Delta)^{-1}$ and, for any $h\in H^{-2}$, let $h_{\varepsilon}=J_{\varepsilon}h$.
By \eqref{solW}, we have
\begin{align}
X_{\varepsilon}(t)&=x_{\varepsilon}-\int_0^t\left[i\Delta X_{\varepsilon}(s)+\frac{1}{2}\sum_{j=1}^N(|\mu_j|^2|e_j|^2V_jX)_{\varepsilon}(s)+\lambda ig_{\varepsilon}(s)\right]ds \nonumber \\
& \quad +\sum_{k=1}^N\int_0^t(X\mu_ke_k)_{\varepsilon}(s)dM_k(s), \ t\in [0,\tau_n], \ \Pas,
\end{align}
where $g_{\varepsilon}(s)=J_{\varepsilon}[|X(s)|^{\alpha-1}X(s)]\in L^2, \ 0\le s<\tau^*(x)$. 

Then for every $f_j$, we obtain
\begin{align}
\label{Mceq2}
\langle f_j,X_{\varepsilon}(t)\rangle_{L^2}&=\langle f_j,x_{\varepsilon}\rangle_{L^2}-\left\langle f_j,\int_0^t\left[i\Delta X_{\varepsilon}(s)+\frac{1}{2}\sum_{\ell=1}^N(|\mu_{\ell}|^2|e_{\ell}|^2V_{\ell}X)_{\varepsilon}(s)+\lambda ig_{\varepsilon}(s)\right]ds \right\rangle_{L^2} \nonumber \\
& \quad +\left\langle f_j,\sum_{k=1}^N\int_0^t(X\mu_ke_k)_{\varepsilon}(s)dM_k(s)\right\rangle_{L^2}, \ t\in[0,\tau_n].
\end{align}
By $X|_{[0,\tau_n]}\in C([0,\tau_n];L^2)\cap L^q(0 \tau_n;L^{\alpha+1})$, the integrals for the drift term can be interchanged. For the stochastic integral in \eqref{Mceq2}, we set
\[ \sigma_{n,m}=\inf\{t\in[0,\tau_n] : \|X(t)\|_{L^2}>m\}\wedge \tau_n. \]
Then, by the estimate $\|J_{\varepsilon}(f)\|_{L^2}\le \|f\|_{L^2}$, we have
\begin{align*}
&\mathbb{E} \int_0^{t\wedge \sigma_{n,m}}\sum_{k=1}^N|\langle f_j,(X\mu_ke_k)_{\varepsilon}(s)\rangle_{L^2}|^2d\langle M_k\rangle(s)  \\
&\le \sum_{k=1}^N\|\mu_ke_k\|^2_{L^{\infty}}\|f_j\|^2_{L^2}\mathbb{E} \left[\|V_k\|_{L^{\infty}(0,T)}\int_0^{t\wedge \sigma_{n,m}}\|X(s)\|^2_{L^2}ds\right] \\
&\le C\sum_{k=1}^N\|\mu_ke_k\|^2_{L^{\infty}}\|f_j\|^2_{L^2}m^2t<\infty.
\end{align*}
Therefore, from stochastic Fubini's theorem, we have
\begin{align}
\label{Fubini}
\left\langle f_j, \sum_{k=1}^N\int_0^t(X\mu_ke_k)_{\varepsilon}(s)dM_k(s)\right\rangle_{L^2}=\sum_{k=1}^N\int_0^t\langle f_j,(X\mu_ke_k)_{\varepsilon}(s)\rangle_{L^2}dM_k(s),
\end{align}
on $\{t\le \sigma_{n,m}\}$.
However, by $X|_{[0,\tau_n]}\in C([0,\tau_n];L^2)\cap L^q(0,\tau_n;L^{\alpha+1})$, for every $\Pas \ \omega\in \Omega$, there exists $m(\omega)\in\mathbb{N}$ such that for all $m\ge m(\omega)$, we have $\sigma_{n,m}=\tau_n$. Therefore, we obtain
\begin{align}
\label{bigcup}
\bigcup_{m\in\mathbb{N}}\{t\le \sigma_{n,m}\}=\{t\le \tau_n\}.
\end{align}
This implies \eqref{Fubini} holds on $\{t\le \tau_n\}$.
Therefore, we get
\begin{align}
\langle f_j,X_{\varepsilon}(t)\rangle_{L^2}&=\langle f_j,x_{\varepsilon}\rangle_{L^2}-\int_0^t\left\langle f_j,i\Delta X_{\varepsilon}(s)+\frac{1}{2}\sum_{\ell=1}^N(|\mu_{\ell}|^2|e_{\ell}|^2V_{\ell}X)_{\varepsilon}(s)+\lambda ig_{\varepsilon}(s)\right\rangle_{L^2}ds \nonumber \\
& \quad +\sum_{k=1}^N\int_0^t \langle f_j,(X\mu_ke_k)_{\varepsilon}(s)\rangle_{L^2}dM_k(s), \ t\in[0,\tau_n], \ \Pas
\end{align}
By Itô's product rule, we calculate
\begin{align*}
\|\langle f_j,X_{\varepsilon}(t)\rangle_{L^2}\|^2_{L^2}&=\|\langle f_j,x_{\varepsilon}\rangle_{L^2}\|^2_{L^2}+2\text{Re}\int_0^t\langle X_{\varepsilon}(s),f_j\rangle_{L^2} d\langle f_j,X_{\varepsilon}(s)\rangle_{L^2} \\
& \quad +\langle \langle X_{\varepsilon}(t),f_j\rangle_{L^2},\langle f_j,X_{\varepsilon}(t)\rangle_{L^2}\rangle \\
&=\|\langle f_j,x_{\varepsilon}\rangle_{L^2}\|^2_{L^2}+2\text{Re}\int_0^t\langle X_{\varepsilon}(s),f_j\rangle_{L^2} \langle f_j,-i\Delta X_{\varepsilon}(s)\rangle_{L^2} ds \\
& \quad +2\text{Re}\int_0^t \langle X_{\varepsilon}(s),f_j\rangle_{L^2}\left\langle f_j,-\frac{1}{2}\sum_{\ell=1}^N(|\mu_{\ell}|^2|e_{\ell}|^2 V_{\ell}X)_{\varepsilon}(s)\right\rangle_{L^2} ds \\
& \quad +2\text{Re}\int_0^t \langle X_{\varepsilon}(s),f_j\rangle_{L^2}\langle f_j,-\lambda i g_{\varepsilon}(s)\rangle_{L^2} ds \\
& \quad +2\text{Re}\sum_{k=1}^N\int_0^t\langle X_{\varepsilon}(s),f_j\rangle_{L^2}\langle f_j,(X\mu_ke_k)_{\varepsilon}(s)\rangle_{L^2}dM_k(s) \\
& \quad +\sum_{m=1}^N\int_0^t |\langle f_j,(X\mu_me_m)_{\varepsilon}(s)\rangle_{L^2}|^2 d\langle M_m\rangle(s), \quad t\in[0,\tau_n].
\end{align*}
Taking the sum over $j$, we have
\begin{align}
\label{Mce}
\|X_{\varepsilon}(t)\|^2_{L^2}&=\|x_{\varepsilon}\|^2_{L^2}+2\text{Re}\int_0^t \left\langle X_{\varepsilon}(s),-\frac{1}{2}\sum_{\ell=1}^N(|\mu_{\ell}|^2|e_{\ell}|^2V_{\ell}X)_{\varepsilon}(s)\right\rangle_{L^2} ds \nonumber \\
& \quad +2\text{Re}\int_0^t \langle X_{\varepsilon}(s),-\lambda i g_{\varepsilon}(s)\rangle_{L^2} ds+2\text{Re}\sum_{k=1}^N\int_0^t\langle X_{\varepsilon}(s),(X\mu_ke_k)_{\varepsilon}(s)\rangle_{L^2}dM_k(s) \nonumber \\
& \quad +\sum_{m=1}^N\int_0^t\|(X\mu_me_m)_{\varepsilon}(s)\|^2_{L^2}d\langle M_m\rangle(s), \quad t\in[0,\tau_n].
\end{align}
Therefore, since for $f\in L^p, p\in(1,\infty)$,
\[ \|J_{\varepsilon}(f)\|_{L^p}\le \|f\|_{L^p} \ \text{and} \ J_{\varepsilon}(f)\to f \ \text{in} \ L^p, \text{as} \ \varepsilon\to 0, \]
we take the limit $\varepsilon\to 0$ in \eqref{Mce}.
Then we notice that the second and fifth terms cancel, and the third term goes to zero.
As a result, we obtain $\Pas$ \eqref{Mc} on $\{t\le \tau_n\}$, which implies that \eqref{Mc} holds on $\{t\le \tau^*(x)\}$ as $\tau_n\to \tau^*(x), \Pas$.

Next, we obtain the a priori estimate \eqref{AE} considering $\displaystyle \sum_{j=1}^N|\mu_j|^2\|e_j\|^2_{L^{\infty}}<\infty$ and $\displaystyle \sum_{j=1}^N\|V_j\|_{L^{\infty}(0,T)}\le C, \Pas$. By the Burkholder-Davis-Gundy and Young inequalities, we have
\begin{align*}
& \mathbb{E}\left[ \sup_{s\in[0,t\wedge\tau_n]}\left| \sum_{j=1}^N\int_0^s\int_{\mathbb{R}^d}\text{Re}(\mu_j)e_j|X(r)|^2d\xi dM_j(r)\right|\right] \\
&\le C\mathbb{E}\left[ \int_0^{t\wedge \tau_n}\sum_{j=1}^N\left( \int_{\mathbb{R}^d}\text{Re}(\mu_j)e_j|X(s)|^2d\xi \right)^2 d\langle M_j\rangle(s) \right]^{\frac{1}{2}} \\
&\le C\mathbb{E}\left[ \sum_{j=1}^N\|V_j\|_{L^{\infty}(0,T)}\int_0^{t\wedge \tau_n}\|X(s)\|^4_{L^2}ds \right]^{\frac{1}{2}} \\
&\le C\mathbb{E}\left[ \sup_{s\in[0,t\wedge\tau_n]}\|X(s)\|_{L^2}\left( \int_0^{t\wedge \tau_n}\|X(s)\|^2_{L^2}ds \right)^{\frac{1}{2}}\right] \\
&\le C\sqrt{\mathbb{E} \sup_{s\in[0,t\wedge \tau_n]}\|X(s)\|^2_{L^2}} \sqrt{\mathbb{E} \int_0^{t\wedge\tau_n}\|X(s)\|^2_{L^2} ds} \\
&\le \frac{1}{4}\mathbb{E} \sup_{s\in[0,t\wedge \tau_n]}\|X(s)\|^2_{L^2}+C\int_0^t \mathbb{E}\left( \sup_{r\in[0,s\wedge \tau_n]}\|X(r)\|^2_{L^2} \right) ds,
\end{align*}
for all $t\in[0,T]$ and all $n\in\mathbb{N}$, where $C$ is a constant independent of $n$ and changes with each line. Combining this with \eqref{Mc}, we obtain
\[ \mathbb{E}\left[ \sup_{s\in[0,t\wedge \tau_n]}\|X(s)\|^2_{L^2}\right]\le 2\|x\|^2_{L^2}+4C\int_0^t\mathbb{E} \left( \sup_{r\in[0,s\wedge \tau_n]}\|X(r)\|^2_{L^2}\right)ds. \]
This implies
\[ \mathbb{E}\left[ \sup_{t\in [0,T\wedge \tau_n]}\|X(t)\|^2_{L^2}\right] \le \widetilde{C}(T), \]
where $\widetilde{C}(T)$ is a constant independent of $n$.
Therefore, we obtain \eqref{AE} as $n\to \infty$. 
\end{proof}
\begin{proof}[Proof of Theorem \ref{main2}]
By Theorem \ref{Global1} and the Blow-up alternative, we immediately show the existence of a global solution. 

We show continuous dependence on the initial data $x\in L^2$. 
Let $x_m\to x$ in $L^2$ as $m\to \infty$. For any $m\ge 1$, let $(y_m,T)$ denote the unique global solution to \eqref{RSNLS2} corresponding to the initial value $x_m$. And, let $(y, T)$ denote the unique global solution to \eqref{RSNLS2} corresponding to the initial value $x$.

Since $\|x_m\|_{L^2}\le \|x\|_{L^2}+1$ for all $m\ge m_1$ with $m_1$ large enough, we can modify the stopping time $\tau_1(\le T)$ in Step 1 in the proof of Theorem \ref{LSL^2} such that
\begin{align*}
\tau_1&=\inf\left\{t\in[0,T], 2C_t\sum_{j=1}^N\||\mu_j|^2|e_j|^2+\mu_j^2e_j^2\|_{L^{\infty}}\|V_j\|_{L^{\infty}(0,T)}t^{\theta} \right. \\
&\left. \hspace{5cm} +2\cdot 3^{\alpha-1}C_{\alpha,t}(\|x\|_{L^2}+1)^{\alpha-1}C_t^{\alpha}\gamma_tt^{\theta}>\frac{1}{3}\right\}\wedge T,
\end{align*}
which is independent for all $m\ge m_1$.
Hence, using similar contraction arguments as in Step 1 in the proof of Theorem \ref{LSL^2} and the uniqueness, we obtain
\[ \widetilde{M}_1:=\sup_{m\ge m_1}[\|y_m\|_{L^{\infty}(0,\tau_1;L^2)}+\|y_m\|_{L^q(0,\tau_1;L^{\alpha+1})}]\le 3C_{\tau_1}(\|x\|_{L^2}+1). \]
We note that for $m\ge m_1$
\begin{align*}
&\|y_m-y\|_{L^{\infty}(0,\tau_1;L^2)}+\|y_m-y\|_{L^q(0,\tau_1;L^{\alpha+1})} \\
&\le 2C_T\|x_m-x\|_{L^2}+\left(4C_{\tau_1}\sum_{j=1}^N\||\mu_j|^2|e_j|^2+\mu_j^2e_j^2\|_{L^{\infty}}\|V_j\|_{L^{\infty}(0,T)}\tau_1^{\theta}\right. \\
&\hspace{5cm} \left. +4C_{\alpha,\tau_1}C_{\tau_1}\gamma_{\tau_1}\tau_1^{\theta}\widetilde{M_1}^{\alpha-1}\right)\|y_m-y\|_{\mathcal{X}_{\tau_1}}.
\end{align*}
By the choice of $\tau_1$ and the bound on $\widetilde{M}_1$, we have
\[ 4C_{\tau_1}\sum_{j=1}^N\||\mu_j|^2|e_j|^2+\mu_j^2e_j^2\|_{L^{\infty}}\|V_j\|_{L^{\infty}(0,T)}\tau_1^{\theta}+4C_{\alpha,\tau_1}C_{\tau_1}\gamma_{\tau_1}\tau_1^{\theta}\widetilde{M_1}^{\alpha-1}\le \frac{2}{3}, \]
hence
\[ \frac{1}{3}\|y_m-y\|_{L^{\infty}(0,\tau_1;L^2)}+\frac{1}{3}\|y_m-y\|_{L^q(0,\tau_1;L^{\alpha+1})}\le 2C_T\|x_m-x\|_{L^2}\to 0, \ \text{as} \ m\to \infty. \]
Thus, we obtain the continuous dependence on the interval $[0,\tau_1]$. Now, since $y_m(\tau_1)\to y(\tau_1)$, using similar arguments as above we can extend the above results to $[0,\tau_2]$ with $\tau_2$ depending on $\|y(\tau_1)\|_{L^2}$ and $\tau_1\le \tau_2\le T$. Reiterating the arguments, we then have an increasing sequence of stopping times $\tau_n$ depending on $\|y(\tau_{n-1})\|_{L^2}$, such that the continuous dependence holds on $[0,\tau_n], n\in \mathbb{N}$. Since $\displaystyle \sup_{t\in[0,T]}\|y(t)\|_{L^2}<1, \Pas$, as in the proof for the blowup alternative in Theorem \ref{LSL^2}, we deduce that for $\Pas \omega$ there exists $n(\omega)<\infty$ such that $\tau_{n(\omega)}(\omega)=T$. Therefore, we obtain the continuous dependence on $[0,T]$.
\end{proof}

\section{Exponential Decay of Solutions}

Next, we show the exponential decay of the solution to \eqref{SNLS}.
Assume (H2) with $r_j=\infty$ for all $j$, (H3), and (H4).
In this case, we express equation \eqref{RSNLS2} as follows:
\begin{align}
\label{RSNLS3}
\begin{cases}
\displaystyle 
i\partial_ty=\Delta y-\frac{i}{2}\sum_{j=1}^N(|\mu_j|^2+\mu^2_j)V_j(s)y+\lambda e^{(\alpha-1)\text{Re}M}|y|^{\alpha-1}y, \\
y(0,\xi)=x(\xi), \quad \xi\in\mathbb{R}^d. 
\end{cases}
\end{align}
To demonstrate the exponential decay of \eqref{SNLS}, we first show the exponential decay of \eqref{RSNLS3}.
Multiplying both sides of \eqref{RSNLS3} by $\overline{y}$, taking the imaginary part, and integrating yields the following energy equation:
\begin{align}
\label{ene}
E_0(t):=\int_{\mathbb{R}^d}|y(t,\xi)|^2d\xi=-2\sum_{j=1}^N\int_0^t\int_{\mathbb{R}^d}(\text{Re} \ \mu_j)^2V_j(s)|y(s,\xi)|^2d\xi ds+\|x\|^2_{L^2}\le \|x\|^2_{L^2}.
\end{align}
Then, we have
\[ \frac{d}{dt}E_0(t):=\frac{d}{dt}\int_{\mathbb{R}^d}|y(t,\xi)|^2d\xi=-2\sum_{j=1}^N\int_{\mathbb{R}^d}(\text{Re} \ \mu_j)^2V_j(t)|y(t,\xi)|^2d\xi. \]

First, we prove the following theorem.
\begin{theorem}
\label{exdecay}
Assume (H2) with $r_j=\infty$ for all $j$, (H3) and (H4), and let $1<\alpha<1+\frac{4}{d}$. 
Then, there exist a strictly positive deterministic constant $\omega>0$ such that
\[ E_0(t)\le e^{-\omega t}E_0(0), \quad \Pas \]
for all $t\ge0$.
\end{theorem}
\begin{proof}
By the energy identity \eqref{ene} and assumption (H4), we obtain
\begin{align*}
\frac{d}{dt} E_0(t)=-2 \left( \sum_{j=1}^N (\text{Re} \ \mu_j)^2 V_j(t) \right) E_0(t)\le -2\alpha_0 \left( \sum_{j=1}^N(\text{Re} \ \mu_j)^2\right) E_0(t).
\end{align*}
Let $\omega:=2\alpha_0 \sum_{j=1}^N(\text{Re} \ \mu_j)^2$. By $\text{Re} \ \mu_j\not=0$, we have $\omega>0$. Applying Gronwall's inequality yields
\begin{equation*}
E_0(t) \le E_0(0) e^{-\omega t} \quad \text{for all } t \ge 0.
\end{equation*}
This completes the proof of the exponential decay for the rescaled solution $y(t)$.
\end{proof}

Finally, we prove Theorem \ref{exdecay2} of the main theorem using Theorem \ref{exdecay}.
\begin{proof}[Proof of Theorem \ref{exdecay2}]
Let $X=e^My$. By Theorem \ref{exdecay}, we have the pathwise estimate
\[ \|X(t)\|^2_{L^2_{\xi}(\mathbb{R}^d)}=e^{2\text{Re} M(t)}\|y(t)\|^2_{L^2_{\xi}(\mathbb{R}^d)}\le \|x\|^2_{L^2}e^{-\omega t}e^{2\text{Re} M(t)}. \]
To investigate the precise exponential decay rate, we take the natural logarithm of both sides and divide by $t>0$, which yields
\begin{align}
\label{540}
\frac{1}{t}\log \|X(t)\|^2_{L^2_{\xi}(\mathbb{R}^d)}\le \frac{\log \|x\|^2_{L^2}}{t}-\omega+2\frac{\text{Re} M(t)}{t}.
\end{align}
Recall that under the assumption (H4), the continuous martingale $\text{Re} M(t)$ is given by $\text{Re} M(t)=\sum_{j=1}^N(\text{Re} \ \mu_j)M_j(t)$. Its quadratic variation is calculated as
\[ \langle \text{Re} M\rangle_t=\sum_{j=1}^N(\text{Re} \ \mu_j)^2\langle M_j\rangle_t=\sum_{j=1}^N(\text{Re} \ \mu_j)^2\int_0^tV_j(s)ds. \]
By (H4) and the fact that $\text{Re} \mu_j \neq 0$ and $V_j(t) \ge \alpha_0 > 0$, it follows that the quadratic variation of the real-valued continuous martingale $\text{Re} M(t)$ satisfies $\langle \text{Re} M \rangle_t \to \infty$ as $t \to \infty$ $P$-a.s. 
Recall that for a continuous square-integrable martingale, the pair $(\text{Re} M, \langle \text{Re} M \rangle)$ satisfies the strong law of large numbers (cf. \cite[p. 144]{LS89}):
\begin{align*}
\lim_{t \to \infty} \frac{\text{Re} M(t)}{\langle \text{Re} M \rangle_t} = 0, \quad \Pas
\end{align*}
Using this property, we can evaluate the term in \eqref{540}. Specifically, we observe that
\begin{align*}
\frac{\text{Re} M(t)}{t} = \frac{\text{Re} M(t)}{\langle \text{Re} M \rangle_t} \cdot \frac{\langle \text{Re} M \rangle_t}{t}.
\end{align*}
By $V_j\in L^{\infty}(0,\infty)$, we can bound the quadratic variation linearly in time. Specifically, there exists a deterministic constant $K>0$ such that $\langle \text{Re} M\rangle_t\le Kt$ for all $t\ge0$, where $K=\sum_{j=1}^N(\text{Re} \ \mu_j)^2\|V_j\|_{L^{\infty}(0,\infty)}$.
Therefore, we obtain
\begin{align}
\label{541}
\lim_{t\to\infty}\frac{\text{Re} M(t)}{t}=0, \quad \Pas
\end{align}
Therefore, taking the limit superior as $t\to \infty$ in \eqref{540} and applying \eqref{541}, we obtain
\[ \limsup_{t\to \infty}\frac{1}{t}\log\|X(t)\|^2_{L^2_{\xi}(\mathbb{R}^d)}\le \limsup_{t\to \infty}\left( \frac{\log \|x\|^2_{L^2}}{t}-\omega+2\frac{\text{Re} M(t)}{t}\right)=-\omega, \quad \Pas. \]
This concludes the proof of Theorem \ref{exdecay2}.
\end{proof}
\begin{rmk}
The exponential decay result established in Theorem \ref{exdecay2} includes the case of standard Brownian motions as a special case. Specifically, if $M_j(t)=B_j(t)$ are independent standard Brownian motions, their quadratic variations satisfy $\langle M_j \rangle_t=t$, which implies $V_j(t)\equiv 1$ for all $t\ge 0$. In this setting, our assumption (H4) is satisfied with $\alpha_0=1$, and the decay rate $\omega=2\sum_{j=1}^N (\text{Re} \mu_j)^2$ is consistent with the existing stabilization results for SNLS driven by standard Brownian motion (cf. \cite{BRZ14}).
\end{rmk}

\section*{Declarations}

\subsection*{Conflicts of interests}
The authors declare that there is no conflict of interest regarding the publication of this paper.

\subsection*{Data Availability Statements}
Data sharing is not applicable to this article as no datasets were generated or analyzed during the current study.


\begin{thebibliography}{99}

\bibitem{BCRG94}
O. Bang, P. L. Christiansen, F. If, K. O. Rasmussen, Y. B. Gaididei,
\textit{Temperature effects in a nonlinear model of monolayer Scheibe aggregates},
Phys. Rev. E., \textbf{49} 4627–4636 (1994).

\bibitem{BCRG95}
O. Bang, P. L. Christiansen, F. If, K. O. Rasmussen, Y. B. Gaididei,
\textit{White noise in the two-dimensional nonlinear Schr$\ddot{\text{o}}$dinger equation},
Appl. Anal., \textbf{57} 3–15 (1995).

\bibitem{BRZ14}
V. Barbu, M. R$\ddot{\text{o}}$ckner, D. Zhang,
\textit{Stochastic nonlinear Schr$\ddot{\text{o}}$dinger equations with linear multiplicative noise: rescaling approach},
J. Nonlinear Sci., \textbf{24} (3) 383-409 (2014).

\bibitem{BRZ16}
V. Barbu, M. R$\ddot{\text{o}}$ckner, D. Zhang,
\textit{Stochastic nonlinear Schr$\ddot{\text{o}}$dinger equations},
Nonlinear Anal., \textbf{136} 168-194 (2016).

\bibitem{BRZ18}
V. Barbu, M. R$\ddot{\text{o}}$ckner, D. Zhang,
\textit{Optimal bilinear control of nonlinear stochastic Schr$\ddot{\text{o}}$dinger equations driven by linear multiplicative noise},
Ann. Probab., \textbf{46} (4) 1957-1999 (2018).

\bibitem{BG09}
A. Barchielli, M. Gregoratti,
\textit{Quantum Trajectories and Measurements in Continuous Time: The Diffusive Case},
Lecture Notes Physics, \textbf{782} Springer, Berlin (2009).

\bibitem{BPP10}
A. Barchielli, C. Pellegrini, F. Petruccione,
\textit{Stochastic Schr$\ddot{\text{o}}$dinger equations with coloured noise},
Lett. J. Explor. Front. Phys. EPL., \textbf{91} (2010).

\bibitem{BD99}
A. de Bouard, A. Debussche, 
\textit{A stochastic nonlinear Schr$\ddot{\text{o}}$dinger equation with multiplicative noise}, 
Comm. Math. Phys., \textbf{205} 161–181 (1999).

\bibitem{BD03}
A. de Bouard, A. Debussche,
\textit{The stochastic nonlinear Schr$\ddot{\text{o}}$dinger equation in $H^1$},
Stoch. Anal. Appl., \textbf{21} 97–126 (2003).

\bibitem{BHM20}
Z. Brze\'zniak, F. Hornung, U. Manna,
\textit{Weak martingale solutions for the stochastic nonlinear Schr$\ddot{\text{o}}$dinger equation driven by pure jump noise},
Stoch. Partial Differ. Equ.: Anal. Comput., \textbf{8} (1) 1-53 (2020).

\bibitem{BHW19}
Z. Brze\'zniak, F. Hornung, L. Weis,
\textit{Martingale solutions for the stochastic nonlinear Schr$\ddot{\text{o}}$dinger equation in the energy space},
Probab. Theory Related Fields, \textbf{174} (3-4) 1273-1338 (2019).

\bibitem{BHW22}
Z. Brze\'zniak, F. Hornung, L. Weis,
\textit{Uniqueness of martingale solutions for the stochastic nonlinear Schr$\ddot{\text{o}}$dinger equation on 3d compact manifolds},
Stoch. Partial Differ. Equ.: Anal. Comput., (2022). https://doi.org/10.1007/s40072-022-00238-w

\bibitem{C03}
T. Cazenave,
\textit{Semilinear Schr\"odinger equations},
Courant Lect. Notes in Math., \textbf{10} (2003).

\bibitem{CW90}
T. Cazenave, F. B. Weissler,
\textit{The Cauchy problem for the critical nonlinear Schr\"odinger equation in $H^s$},
Nonlinear Anal., \textbf{14} 807-836 (1990).

\bibitem{DHM25}
I. D\^oku, S. Hashimoto, S. Machihara,
\textit{The well-posedness of the stochastic nonlinear Schr$\ddot{\text{o}}$dinger equations in $H^2(\mathbb{R}^d)$},
Advances in Differential Equations, \textbf{30} (7-8) 527-560 (2025).

\bibitem{GV79}
J. Ginibre, G. Velo,
\textit{On a class of nonlinear Schr$\ddot{\text{o}}$dinger equations},
J. Funct. Anal., \textbf{32} 1-71 (1979).

\bibitem{HHM24}
M. Hamano, S. Hashimoto, S. Machihara,
\textit{Global and local solutions of stochastic nonlinear Schr$\ddot{\text{o}}$dinger system with quadratic interaction},
Bull. Iran. Math. Soc., \textbf{50} (22) (2024).

\bibitem{HHM25}
M. Hamano, S. Hashimoto, S. Machihara,
\textit{Global solution for the stochastic nonlinear Schr$\ddot{\text{o}}$dinger system with quadratic interaction in four dimensions},
Saitama Math. J., \textbf{36} 1-13 (2025).

\bibitem{HRZ19}
S. Herr, M. R$\ddot{\text{o}}$ckner, D. Zhang,
\textit{Scattering for stochastic nonlinear Schr$\ddot{\text{o}}$dinger equations},
Comm. Math. Phys., \textbf{368} (2) 843-884 (2019).

\bibitem{H20}
F. Hornung,
\textit{The stochastic nonlinear Schr$\ddot{\text{o}}$dinger equation in unbounded domains and non-compact manifolds},
Nonlinear Differ. Equ. Appl., \textbf{27} 40 (2020).

\bibitem{K87}
T. Kato,
\textit{On nonlinear Schr$\ddot{\text{o}}$dinger equations},
Ann. Inst. H. Poincar$\acute{\text{e}}$ Phys. Th$\acute{\text{e}}$or., \textbf{46} 113-129 (1987).

\bibitem{K89}
T. Kato,
\textit{Nonlinear Schr$\ddot{\text{o}}$dinger Equations, Schr$\ddot{\text{o}}$dinger Operators, (S$\phi$nderberg, 1988)},
Lecture Notes in Physics, \textbf{345} Springer, Berlin 218–263 (1989).

\bibitem{LS89}
R. Sh. Liptser, A. N. Shiryayev,
\textit{Theory of Martingales},
Mathematics and Its Applications, Kluwer Academic Publishers, (1989).

\bibitem{N15}
F. Natali,
\textit{Exponential Stabilization for the Nonlinear Schr$\ddot{\text{o}}$dinger Equation with Localized Damping},
J. Dyn. Control Syst., \textbf{21} 461-474 (2015).

\bibitem{T84}
M. Tsutsumi,
\textit{Nonexistence of global solutions to the Cauchy problem for the damped nonlinear Schr$\ddot{\text{o}}$dinger equations},
SIAM, J., Math., Anal., \textbf{15} (2) (1984).

\bibitem{T87}
Y. Tsutsumi,
\textit{$L^2$-solutions for nonlinear Schr$\ddot{\text{o}}$dinger equations and nonlinear groups},
Funkcial. Ekvac., \textbf{30} 115-125 (1987).

\end{thebibliography}
\end{document}